\documentclass
{amsart}
\usepackage{graphicx}
\newtheorem{thm}{Theorem}[section]
\theoremstyle{definition}
\newtheorem{cor}[thm]{Corollary}
\newtheorem{prop}[thm]{Proposition}
\newtheorem{defn}[thm]{Definition}
\newtheorem{lem}[thm]{Lemma}

\newtheorem{rem}[thm]{Remark}
\newtheorem{ex}[thm]{Example}

\numberwithin{equation}{section}
\begin{document}
\title{Some generalizations of second submodules}

\author{H. Ansari-Toroghy}
\address {Department of pure Mathematics, Faculty of mathematical
Sciences, University of Guilan,
P. O. Box 41335-19141, Rasht, Iran.}%
\email{ansari@guilan.ac.ir}%

\author{F. Farshadifar}
\footnote {This research was in part supported by a grant from IPM (No. 94130048)}
\address{University of Farhangian, P. O. Box 19396-14464, Tehran, Iran.}
\email{f.farshadifar@gmail.com}

\address{School of Mathematics, Institute for Research in Fundamental Sciences (IPM), P.O. Box: 19395-5746, Tehran, Iran.}

\subjclass[2000]{13C13, 13C99}%
\keywords {
2-absorbing second, completely irreducible, and  strongly 2-absorbing second.}

\begin{abstract}
In this paper, we will introduce two generalizations of second submodules of a module over a commutative ring and explore some basic properties of these classes of modules
\end{abstract}

\maketitle
% ----------------------------------------------------------------
\section{Introduction}
\noindent
Throughout this paper, $R$ will denote a commutative ring with
identity and "$\subset$" will denote the strict inclusion.  Further, $\Bbb Z$ will denote the ring of integers.

Let $M$ be an $R$-module. A proper submodule $P$ of $M$ is said
 to be \emph{prime} if for any $r \in R$ and $m \in M$ with
$rm \in P$, we have $m \in P$ or $r \in (P:_RM)$ \cite{Da78}.
A non-zero submodule $S$ of $M$ is said to be \emph{second}
if for each $a \in R$, the homomorphism $ S \stackrel {a} \rightarrow S$
is either surjective or zero \cite{Y01}.
In this case $Ann_R(S)$ is a prime ideal of $R$.

Badawi gave a generalization of prime ideals in \cite{Ba07} and said such ideals 2-
absorbing ideals. A proper ideal $I$ of $R$ is a \emph{2-absorbing ideal}
of $R$ if whenever $a, b, c \in R$ and $abc \in I$, then $ab \in I$ or
 $ac \in I$ or $bc \in I$. He proved that $I$ is a 2-absorbing ideal of
 $R$ if and only if whenever $I_1$, $I_2$, and $I_3$ are ideals of $R$
with $I_1I_2I_3 \subseteq I$,
then $I_1I_2 \subseteq I$ or $I_1I_3 \subseteq I$ or $I_2I_3 \subseteq I$.
Yousefian Darani and Soheilnia in \cite{YS11} extended 2-absorbing ideals
 to 2-absorbing submodules.
A proper submodule $N$ of $M$ is called a \emph{2-absorbing submodule }of
 $M$ if whenever $abm \in N$
for some $a, b \in R$ and $m \in M$, then $am \in N$ or $bm \in N$ or
$ab \in (N :_R M)$. Several authors investigated properties of 2-absorbing submodules, for example see \cite{YS11, pb12, pb15}.

A submodule $N$ of an $R$-module $M$ is called \emph{strongly 2-absorbing} if $IJL \subseteq N$ for some ideals $I, J$ of $R$ and a submodule $L$ of $M$, then
$IL \subseteq N$ or $JL \subseteq N$ or $IJ \in (N :_R M)$ \cite{DS12}.

The purpose of this paper is to introduce the dual notions of 2-absorbing and strongly
2-absorbing submodules and obtain some related results. Also, as we can see in Corollary \ref{c2.18}, these are two generalizations of second submodules. In \cite[2.3]{pb15}, the authors show that $N$ is a 2-absorbing submodule
of an $R$-module $M$ if and only if $N$ is a strongly 2-absorbing submodule of $M$. The Example \ref{e2.2} shows that the dual of this fact is not true in general.

\section{2-absorbing second submodules}
\noindent
Let $M$ be an $R$-module. A proper submodule $N$ of
$M$ is said to be \emph{completely irreducible} if $N=\bigcap _
{i \in I}N_i$, where $ \{ N_i \}_{i \in I}$ is a family of
submodules of $M$, implies that $N=N_i$ for some $i \in I$. It is
easy to see that every submodule of $M$ is an intersection of
completely irreducible submodules of $M$ \cite{FHo06}.

We frequently use the following basic fact without further comment.
\begin{rem}\label{r2.1}
Let $N$ and $K$ be two submodules of an $R$-module $M$. To prove $N\subseteq K$, it is enough to show that if $L$ is a completely irreducible submodule of $M$ such that $K\subseteq L$, then $N\subseteq L$.
\end{rem}

\begin{defn}\label{d2.1}
Let $N$ be a non-zero submodule of an $R$-module $M$. We say that
$N$ is a \emph{2-absorbing second submodule of} $M$ if whenever
 $a, b \in R$, $L$ is a completely irreducible submodule of $M$,
and $abN\subseteq L$, then $aN\subseteq L$ or
$bN\subseteq L$ or $ab \in Ann_R(N)$. This can be regarded as a dual notion of the
2-absorbing submodule.
\end{defn}

A non-zero $R$-module $M$ is said to be \emph{secondary} if for each $a \in R$ the endomorphism of $M$ given by multiplication by $a$ is either surjective or nilpotent \cite{Ma73}.

\begin{thm}\label{t2.1} Let $M$ be an $R$-module. Then we have the following.
\begin{itemize}
  \item [(a)] If either $N$ is a second submodule of $M$ or N is a sum of two second
submodules of $M$, then $N$ is 2-absorbing second.
  \item [(b)] If $N$ is a secondary submodule of $M$ and $R/Ann_R(N)$ has
no non-zero nilpotent element, then $N$ is 2-absorbing second.
\end{itemize}
\end{thm}
\begin{proof}
(a) The first assertion is clear. To see the second assertion, let $N_1$ and $N_2$ be two second submodules of
$M$. We show that $N_1+ N_2$ is a 2-absorbing second submodule of $M$. Assume that
$a, b \in R$, $L$ is a completely irreducible submodule of $M$,
and $ab(N_1+N_2)\subseteq L$. Since $N_1$ is second, $abN_1=0$
or $N_1\subseteq L$ by \cite[2.10]{AF12}. Similarly, $abN_2=0$
or $N_2\subseteq L$. If $abN_1=0=abN_2$ (resp. $N_1\subseteq L$ and $N_2\subseteq L$), then we are done. Now let $abN_1=0$ and $N_2\subseteq L$.
Then $aN_1=0$ or $bN_1=0$ because $Ann_R(N_1)$ is a prime
ideal of $R$. If $aN_1=0$, then $a(N_1+N_2) \subseteq aN_1+N_2 \subseteq N_2 \subseteq L$. Similarly, if $bN_1=0$, we get  $b(N_1+N_2)\subseteq L$ as desired.

(b) Let $a, b \in R$, $L$ be a completely irreducible
submodule of $M$, and $abN\subseteq L$. Then
if $aN\subseteq L$ or $bN\subseteq L$, we are done.
Let $aN\not \subseteq L$ and $bN\not \subseteq L$.
Then $a, b \in \sqrt{Ann_R(N)}$. Thus, $(ab)^s \in Ann_R(N)$ for some positive integer $s$.
Therefore, $ab \in Ann_R(N)$ because $R/Ann_R(N)$ has no non-zero nilpotent element.
\end{proof}

\begin{lem}\label{l1.3}
Let $I$ be an ideal of $R$ and $N$ be a 2-absorbing second submodule of $M$.
If $a\in R$, $L$ is a completely irreducible submodule of $M$,
and $IaN \subseteq L$, then $aN \subseteq L$ or $IN \subseteq L$ or $Ia \in Ann_R(N)$.
\end{lem}
\begin{proof}
Let $aN \not \subseteq L$ and $Ia \not \in Ann_R(N)$.
Then there exists $b \in I$ such that $abN \not = 0$. Now as $N$
is a 2-absorbing second submodule of $M$, $baN \subseteq L$ implies that $bN \subseteq L$. We show that $IN \subseteq L$. To see this, let $c$ be an arbitrary element of $I$.
Then $(b + c)aN \subseteq L$. Hence, either $(b + c)N \subseteq L$ or $(b + c)a \in Ann_R(N)$. If $(b+c)N \subseteq L$, then since $bN \subseteq L$ we have $cN \subseteq L$. If $(b+c)a \in Ann_R(N)$, then $ca \not \in Ann_R(N)$, but $caN \subseteq L$. Thus $cN \subseteq L$. Hence, we conclude that
$IN \subseteq L$.
\end{proof}

\begin{lem}\label{l1.4}
Let $I$ and $J$ be two ideals of $R$ and $N$ be a 2-absorbing second submodule of $M$.
If $L$ is a completely irreducible submodule of $M$ and $IJN \subseteq L$,
then $IN \subseteq L$ or $JN \subseteq L$ or $IJ \subseteq Ann_R(N)$.
\end{lem}
\begin{proof}
Let $IN \not \subseteq L$ and $JN \not \subseteq L$. We show that $IJ \subseteq Ann_R(N)$. Assume that $c \in I$ and $d\in J$. By assumption there exists $a\in I$ such that
$aN \not \subseteq L$ but $aJN \subseteq L$. Now Lemma \ref{l1.4}
shows that $aJ \subseteq Ann_R(N)$ and so $(I\setminus (L:_RN))J\subseteq Ann_R(N)$.
Similarly there exists $b \in (J\setminus (L:_RN))$ such that $Ib \subseteq Ann_R(N)$ and also $I(J\setminus (L:_RN))\subseteq Ann_R(N)$. Thus we have $ab \in Ann_R(N)$,
$ad \in Ann_R(N)$ and $cb \in Ann_R(N)$. As $a + c \in I$
and $b + d \in J$, we have $(a + c)(b + d)N \subseteq L$.
Therefore, $(a + c)N \subseteq L$ or $(b + d)N \subseteq L$ or $(a + c)(b + d) \in Ann_R(N)$. If $(a + c)N \subseteq L$, then $cN \not \subseteq L$. Hence $c \in  I\setminus (L:_RN)$
which implies that $cd \in Ann_R(N)$. Similarly if $(b + d)N \subseteq L$, we can deduce
that $cd \in Ann_R(N)$. Finally if $(a+c)(b+d) \in Ann_R(N)$, then $ab+ad+cb+cd \in Ann_R(N)$
so that $cd \in Ann_R(N)$. Therefore, $IJ \subseteq Ann_R(N)$.
\end{proof}

\begin{cor}\label{c1.1}
Let $M$ be an $R$-module and $N$ be a 2-absorbing second submodule of
$M$. Then $IN$ is a 2-absorbing second submodules of $M$ for all ideals $I$ of $R$ with $I \not \subseteq Ann_R(N)$.
\end{cor}
\begin{proof}
Let $I$ be an ideal of $R$ with $I \not \subseteq Ann_R(N)$, $a, b \in R$, $L$ be a completely irreducible submodule of $M$,
and $abIN\subseteq L$. Then $aN \subseteq L$ or $bIN \subseteq L$ or $abIN=0$
by Lemma \ref{l1.3}. If $bIN \subseteq L$ or $abIN=0$, then we are done.
If $aN \subseteq L$, then $aIN \subseteq aN$ implies that $aIN \subseteq L$, as needed.
\end{proof}

An $R$-module $M$ is said to be a \emph{multiplication module} if for every submodule $N$ of $M$ there exists an ideal $I$ of $R$ such that $N=IM$ \cite{Ba81}.
\begin{cor}\label{c11.1}
Let $M$ be a multiplication 2-absorbing second $R$-module. Then every non-zero  submodule of $M$ is a 2-absorbing second submodule of $M$.
\end{cor}
\begin{proof}
This follows from Corollary \ref{c1.1}.
\end{proof}

The following example shows that the condition ``$M$ is a multiplication module" in Corollary \ref{c11.1} can not be omitted.
\begin{ex}\label{e2.1}
For any prime integer $p$, let $M=\Bbb Z_{p^\infty}$ and
 $N=\langle 1/p^3+\Bbb Z\rangle$. Then clearly, $M$ is a 2-absorbing second $\Bbb Z$-module but $p^2\langle 1/p^3+\Bbb Z\rangle \subseteq \langle 1/p+\Bbb Z\rangle$
implies that $N$ is not a 2-absorbing second submodule of $M$.
\end{ex}

We recall that an $R$-module $M$ is said to be a \emph{cocyclic module} if
$Soc_R(M)$ is a large and simple submodule of $M$ \cite{Y98}. (Here $Soc_R(M)$
denotes the sum of all minimal submodules of $M$.). A submodule $L$ of $M$ is a completely irreducible submodule of $M$ if and only if $M/L$ is a cocyclic $R$-module \cite{FHo06}.

\begin{prop}\label{c11.6}
Let $N$ be a 2-absorbing second submodule of an $R$-module $M$.
Then we have the following.
\begin{itemize}
   \item [(a)] If $L$ is a completely irreducible submodule of $M$ such that $N \not \subseteq L$, then $(L:_RN)$ is a 2-absorbing ideal of $R$.
   \item [(b)] If $M$ is a cocyclic module, then $Ann_R(N)$ is a 2-absorbing ideal of $R$.
   \item [(c)] If $a \in R$, then  $a^nN=a^{n+1}N$, for all $n \geq 2$.
   \item [(d)] If  $Ann_R(N)$ is a prime ideal of $R$, then $(L:_RN)$ is a prime ideal of $R$ for all completely irreducible submodules $L$ of $M$ such that $N \not \subseteq L$.
  \end{itemize}
\end{prop}
\begin{proof}
(a) Since $N \not \subseteq L$, we have $(L:_RN) \not =R$. Let $a, b, c \in R$ and $abc\in (L:_RN)$.
Then $abN \subseteq (L:_Mc)$. Thus $aN \subseteq (L:_Mc)$  or $bN \subseteq (L:_Mc)$ or $abN=0$ because by \cite[2.1]{AFS124}, $(L:_Mc)$ is a completely irreducible submodule of $M$.  Therefore, $ac\in (L:_RN)$ or $bc\in (L:_RN)$ or $ab\in (L:_RN)$.

(b) Since $M$ is cocyclic, the zero submodule of $M$ is a completely irreducible submodule of $M$. Thus the result follows from part (a).

(c) It is enough to show that $a^2N=a^3N$.
It is clear that $a^3N \subseteq a^2N$. Let $L$ be a completely irreducible submodule of $M$ such that $a^3N \subseteq L$. Then $a^2N \subseteq (L:_Ra)$.
Since $N$ is 2-absorbing second submodule and $(L:_Ra)$ is a completely irreducible submodule of $M$ by \cite[2.1]{AFS124}, $aN \subseteq (L:_Ra)$ or $a^2N=0$. Therefore, $a^2N \subseteq L$. This implies that $a^2N \subseteq a^3N$.

(d) Let $a, b \in R$, $L$ be a completely irreducible
submodule of $M$ such that $N \not \subseteq L$, and $ab \in (L:_RN)$.
Then $aN \subseteq L$ or $bN \subseteq L$ or $abN=0$. If $abN=0$, then by
 assumption, $aN=0$ or $bN=0$. Thus in any cases we get that, $aN \subseteq L$
or $bN \subseteq L$.
\end{proof}

\begin{thm}\label{t11.10}
Let $N$ be a 2-absorbing second submodule of an $R$-module $M$. Then we have the following.
\begin{itemize}
   \item [(a)] If $\sqrt{Ann_R(N)} =P$ for some prime ideal $P$ of $R$ and $L$ is a completely irreducible submodule of $M$ such that $N \not \subseteq L$, then $\sqrt{(L :_R N)}$ is a prime ideal of $R$ containing $P$.
   \item [(b)] If $\sqrt{Ann_R(N)} =P\cap Q$ for some prime ideals $P$ and $Q$ of $R$, $L$ is a completely irreducible submodule of $M$ such that $N \not \subseteq L$, and $P \subseteq \sqrt{(L :_R N)}$, then $\sqrt{(L :_R N)}$ is a prime ideal of $R$.
\end{itemize}
\end{thm}
\begin{proof}
(a) Assume that $a, b \in R$ and $ab \in \sqrt{(L :_R N)}$. Then there is a positive integer $t$ such that $a^tb^tN \subseteq L$. By hypotheses, $N$ is a 2-absorbing second submodule of $M$, thus $a^tN \subseteq L$ or $b^tN \subseteq L$ or $a^tb^t \in Annn_R(N)$. If either $a^tN \subseteq L$ or $b^tN \subseteq L$, we
are done. So assume that $a^tb^t \in Ann_R(N)$. Then $ab \in \sqrt{Ann_R(N)}=P$ and
so $a \in P$ or $b \in P$. It is clear that $P =\sqrt{Ann_R(N)} \subseteq \sqrt{(L :_R N)}$. Therefore, $a \in \sqrt{(L :_R N)}$ or $b \in \sqrt{(L :_R N)}$.

(b) The proof is similar to that of part (a).
\end{proof}

\begin{prop}\label{p7.21}
Let $M$ be an $R$-module and let $\{K_i\}_{i \in I}$ be a chain of
2-absorbing second submodules of $M$. Then $\cup_{i \in I}K_i$ is a 2-absorbing second submodule of $M$.
\end{prop}
\begin{proof}
Let $a, b \in R$, $L$ be a completely irreducible submodule of $M$, and $ab(\cup_{i \in I}K_i) \subseteq L$. Assume that $a(\cup_{i \in I}K_i )\not \subseteq L$ and $b(\cup_{i \in I}K_i) \not \subseteq L$. Then there are $m,n \in I$, where $aK_n \not \subseteq L$ and $bK_m \not \subseteq L$. Hence, for every $K_n \subseteq K_s$ and $K_m \subseteq K_d$ we have $aK_s \not \subseteq L$ and $bK_d \not \subseteq L$.  Therefore, for each submodule $K_h$ such that $K_n \subseteq K_h$ and $K_m \subseteq K_h$ we have $abK_h=0$. Hence $ab(\cup_{i \in I}K_i)=0$, as needed.
\end{proof}

\begin{defn}\label{d7.22}
We say that a 2-absorbing second submodule $N$ of an $R$-module $M$
is a \emph {maximal 2-absorbing second submodule} of a submodule
$K$ of $M$, if $N \subseteq K$ and there does not exist a 2-absorbing second submodule $H$ of $M$ such that $N \subset H \subset K$.
\end{defn}

\begin{lem}\label{l7.23}
 Let $M$ be an $R$-module. Then every 2-absorbing second submodule of $M$ is contained in a maximal 2-absorbing second submodule of $M$.
\end{lem}
\begin{proof}
This is proved easily by using Zorn's Lemma and Proposition \ref{p7.21}.
\end{proof}

\begin{thm}\label{c3.10}
Every Artinian $R$-module has only a finite
number of maximal 2-absorbing second submodules.
\end{thm}
\begin{proof}
Suppose that there exists a non-zero submodule $N$ of $M$
such that it has an infinite number of maximal 2-absorbing second submodules. Let $S$ be a submodule of $M$ chosen minimal such that $S$ has an infinite number of maximal 2-absorbing second submodules. Then $S$ is not 2-absorbing second submodule. Thus there exist $a, b \in R$ and a completely irreducible submodule $L$ of $M$
such that $abS\subseteq L$ but $aS \not \subseteq L$, $bS \not \subseteq L$, and $abS \not=0$. Let $V$ be a maximal 2-absorbing second submodule of $M$ contained in $S$. Then $aV \subseteq L$ or $bV \subseteq L$ or $abV=0$. Thus $V \subseteq (L:_Ma)$  or $V \subseteq (L:_Mb)$ or $V \subseteq (0:_Mab)$. Therefore, $V \subseteq (L:_Sa)$  or $V \subseteq (L:_Sb)$ or $V \subseteq (0:_Sab)$. By the choice of $S$, the modules $(L:_Sa)$, $(L:_Sb)$, and $(0:_Sab)$ have only finitely many maximal 2-absorbing second submodules. Therefore, there is only a finite number of possibilities for the module $S$ which is a contradiction.
\end{proof}

\section{Strongly 2-absorbing second submodules}
\begin{defn}\label{d2.2}
Let $N$ be a non-zero submodule of an $R$-module $M$.
We say that $N$ is a \emph{strongly 2-absorbing second submodule of} $M$
if whenever $a, b \in R$, $L_1, L_2$ are completely irreducible submodules
of $M$, and $abN\subseteq L_1 \cap L_2$, then $aN\subseteq L_1 \cap L_2$ or
$bN\subseteq L_1 \cap L_2$ or $ab \in Ann_R(N)$. This can be regarded as a dual notion of the strongly 2-absorbing submodule.
\end{defn}

\begin{ex}\label{e2.2}
Clearly every strongly 2-absorbing second submodule is a 2-absorbing second submodule. But the converse is not true in general. For example, consider $\Bbb Z$ as a $\Bbb Z$-module. Then $2\Bbb Z$ is a
2-absorbing second submodule of $\Bbb Z$ but it is not  a
strongly 2-absorbing second submodule of $\Bbb Z$.
\end{ex}

\begin{thm}\label{t1.5}
Let $N$ be a submodule of an $R$-module $M$. The following statements
are equivalent:
\begin{itemize}
  \item [(a)] $N$ is a strongly 2-absorbing second submodule of $M$;
  \item [(b)] If $N \not =0$, $IJN \subseteq K$ for some ideals $I, J$ of $R$ and a
submodule $K$ of $M$, then $IN \subseteq K$ or $JN \subseteq K$ or $IJ \in Ann_R(N)$;
  \item [(c)] $N \not =0$ and for each $a , b \in R$, we have $abN=aN$ or $abN=bN$ or $abN=0$.
\end{itemize}
\end{thm}
\begin{proof}
$(a) \Rightarrow (b)$. Assume that $IJN \subseteq K$ for some ideals $I, J$ of $R$, a submodule $K$ of $M$, and $IJ \not \subseteq Ann_R(N)$.
Then by Lemma \ref{l1.4}, for all completely irreducible submodules $L$ of $M$ with $K \subseteq L$ either $IN \subseteq L$ or $JN\subseteq L$. If $IN \subseteq L$ (resp. $JN \subseteq L$) for all completely irreducible submodules $L$ of $M$ with $K \subseteq L$, we are done. Now suppose that $L_1$ and $L_2$
are two completely irreducible submodules of $M$ with $K \subseteq L_1$, $K \subseteq L_2$, $IN \not \subseteq L_1$, and $JN \not \subseteq L_2$.
Then $IN  \subseteq L_2$
and $JN  \subseteq L_1$. Since $IJN \subseteq L_1 \cap L_2$, we
have either $IN \subseteq L_1 \cap L_2$ or $JN \subseteq L_1 \cap L_2$.
As $IN \subseteq L_1 \cap L_2$, we have $IN \subseteq L_1$ which is a
contradiction. Similarly from $JN \subseteq L_1 \cap L_2$ we get a contradiction.

$(b) \Rightarrow (a)$. This is clear.

$(a) \Rightarrow (c)$. By part (a), $N \not =0$. Let $a, b \in R$. Then $abN \subseteq abN$ implies that $aN \subseteq abN$ or $bN \subseteq abN$ or $abN=0$. Thus $abN=aN$ or $abN=bN$ or $abN=0$.

$(c) \Rightarrow (a)$. This is clear.
\end{proof}

\begin{lem}
Let $M$ be an $R$-module, $N\subset K$ be two submodules of $M$, and $K$ be a strongly 2-absorbing second submodule of $M$. Then $K/N$ is a strongly 2-absorbing second submodule of $M/N$.
\end{lem}
\begin{proof}
This is straightforward.
\end{proof}

\begin{prop}\label{p1.6}
Let $N$ be a strongly 2-absorbing second submodule of an $R$-module $M$.
Then we have the following.
\begin{itemize}
  \item [(a)] $Ann_R(N)$ is a 2-absorbing ideal of $R$.
  \item [(b)] If $K$ is a submodule of $M$ such that $N \not \subseteq K$,
   then $(K:_RN)$ is a 2-absorbing ideal of $R$.
  \item [(c)] If $I$ is an ideal of $R$, then  $I^nN=I^{n+1}N$, for all $n \geq 2$.
 \item [(d)] If $(L_1 \cap L_2:_RN)$ is a prime ideal of $R$ for
   all completely irreducible submodules $L_1$ and $L_2$ of $M$
   such that $N \not \subseteq L_1 \cap L_2$, then $Ann_R(N)$ is a prime ideal
   of $R$.
\end{itemize}
\end{prop}
\begin{proof}
(a) Let $a, b, c \in R$ and $abc\in Ann_R(N)$.
Then $abN \subseteq abN$ implies that $aN \subseteq abN$ or $bN \subseteq abN$ or $abN=0$
by Theorem \ref{t1.5}. If  $abN=0$, then we are done. If $aN \subseteq abN$,
then $caN \subseteq cabN=0$. In other case, we do the same.

(b) Let $a, b, c \in R$ and $abc\in (K:_RN)$.
Then $acN \subseteq K$  or $bcN \subseteq K$ or $abcN=0$.
If $acN \subseteq K$  or $bcN \subseteq K$, then we are done.
If $abcN=0$, then the result follows from part (a).

(c) It is enough to show that $I^2N=I^3N$.
It is clear that $I^3N \subseteq I^2N$.
Since $N$ is strongly 2-absorbing second submodule, $I^3N \subseteq I^3N$
implies that $I^2N \subseteq I^3N$ or $IN \subseteq I^3N$ or $I^3N=0$ by
Theorem \ref{t1.5}. If $I^2N \subseteq I^3N$ or $IN \subseteq I^3N$, then
we are done. If $I^3N=0$, then the result follows from part (a).

(d) Suppose that $a, b \in R$ and $abN=0$.
Assume contrary that $aN \not =0$ and $bN \not =0$.
Then there exist completely irreducible submodules $L_1$ and $L_2$ of $M$
such that $aN \not \subseteq L_1$ and $bN \not \subseteq L_2$.
Now since $(L_1 \cap L_2:_RN)$ is a prime ideal of $R$,
$0=abN \subseteq L_1 \cap L_2$ implies
that $bN\subseteq L_1 \cap L_2$ or $aN\subseteq L_1 \cap L_2$.
In any cases, we have a contradiction.
\end{proof}

\begin{rem}\label{r1.7} (\cite[Theorem 2.4]{Ba07}).
 If $I$ is a 2-absorbing ideal of $R$, then one of the following statements must hold:
\begin{itemize}
  \item [(a)] $\sqrt{I} = P$ is a prime ideal of $R$ such that $P^2 \subseteq I$;
  \item [(b)] $\sqrt{I} =  P \cap Q$, $PQ \subseteq I$,
and $\sqrt{I}^2 \subseteq I$ where $P$ and $Q$ are the only distinct
prime ideals of $R$ that are minimal over $I$.
\end{itemize}
\end{rem}

\begin{thm}\label{t1.8}
If $N$ is a strongly 2-absorbing second submodule of $M$
and $N \not \subseteq K$, then either
$(K:_RN)$ is a prime ideal of $R$ or there exists an
element $a\in R$ such that $(K:_RaN)$
is a prime ideal of $R$.
\end{thm}
\begin{proof}
By Preposition \ref{p1.6} and Remark \ref{r1.7}, we have one of the following two cases.
 \begin{enumerate}
   \item [(a)]
 Let $\sqrt{Ann_R(N)} = P$, where $P$ is a prime ideal of $R$. We show that $(K:_RN)$
is a prime ideal of $R$ when $P \subseteq (K:_RN)$.
Assume that $a, b \in R$ and $ab \in (K:_RN)$. Hence $aN \subseteq K$ or $bN \subseteq K$ or $ab \in Ann_R(N)$. If either $aN \subseteq K$ or $bN \subseteq K$, we are done. Now assume that $ab \in Ann_R(N)$. Then $ab \in P$ and so
$a\in P$ or $b\in P$. Thus, $a\in (K:_RN)$ or $b\in (K:_RN)$ and
the assertion follows. If $P \not \subseteq (K:_RN)$, then there exists $a \in P$ such that $aN \not \subseteq K$. By Remark \ref{r1.7}, $P^2 \subseteq Ann_R(N) \subseteq (K:_RN)$, thus
$P \subseteq  (K:_RaN)$. Now a similar argument shows
that $(K:_RaN)$ is a prime ideal of $R$.
  \item [(b)]
Let $\sqrt{Ann_R(N)} =  P \cap Q$, where $P$ and $Q$ are
distinct prime ideals of $R$. If
$P \subseteq (K:_RN)$, then the result follows by a similar proof to that of part (a).
Assume that $P \not  \subseteq (K:_RN)$. Then there exists $a \in P$
such that $aN \not \subseteq  K$.  By Remark \ref{r1.7},
we have $PQ \subseteq Ann_R(N) \subseteq  (K:_RN)$.
Thus, $Q \subseteq (K:_RaN)$ and the result follows by a
similar proof to that of part (a).
 \end{enumerate}
 \end{proof}

Let $M$ be an $R$-module. A prime ideal $P$ of $R$ is said to be a
\emph{coassociated prime} of $M$ if there exists a cocyclic
homomorphic image $T$ of $M$ such that $P=Ann_R(T)$.  The set of all
coassociated prime ideals of $M$ is denoted by $Coass_R(M)$
\cite{Y97}.
\begin{thm}\label{t1.9}
Let $N$ be a strongly 2-absorbing second submodule of an $R$-module $M$. Then we have the following.
\begin{itemize}
 \item [(a)] If $\sqrt{Ann_R(N)} =P$ for some prime ideal $P$ of $R$, $L_1$ and $L_2$ are completely irreducible submodules of $M$ such that $N \not \subseteq L_1$, and $N \not \subseteq L_2$, then either $\sqrt{(L_1 :_R N)} \subseteq \sqrt{(L_2 :_R N)}$ or $\sqrt{(L_2 :_R N)} \subseteq \sqrt{(L_1 :_R N)}$. Hence, $Coass_R(N)$ is a totally ordered set.

 \item [(b)] If $\sqrt{Ann_R(N)} =P\cap Q$ for some prime ideals $P$ and $Q$ of $R$, $L_1$ and $L_2$ are completely irreducible submodules of $M$ such that $N \not \subseteq L_1$ and $N \not \subseteq L_2$, and $P \subseteq \sqrt{(L_1 :_R N)}\cap\sqrt{(L_2 :_R N)}$, then either $\sqrt{(L_1 :_R N)} \subseteq \sqrt{(L_2 :_R N)}$ or $\sqrt{(L_2 :_R N)} \subseteq \sqrt{(L_1 :_R N)}$.  Hence, $Coass_R(N)$ is the union of two totally ordered sets.
\end{itemize}
\end{thm}
\begin{proof}
(a) Assume that $\sqrt{(L_1 :_R N)} \not \subseteq \sqrt{(L_2 :_R N)}$. We show that $\sqrt{(L_2 :_R N)} \subseteq \sqrt{(L_2 :_R N)}$. Suppose that $a \in \sqrt{(L_1 :_R N)}$ and $b \in \sqrt{(L_2 :_R N)}$. Then there exists a positive integer $s$ such that $a^sN \subseteq L_1$, $b^sN \subseteq L_2$, and $b^sN \not \subseteq L_1$.
If $a^sN \subseteq L_1 \cap L_2$, then $a^sN \subseteq L_2$ and so $a \in  \sqrt{(L_2 :_R N)}$. Now assume that $a^sN \not \subseteq L_1 \cap L_2$. Then $a^sb^s \in Ann_R(N)$ because $a^sb^sN \subseteq L_1\cap L_2$, $a^sN \not \subseteq L_1 \cap \L_2$, and $b^sN \not \subseteq L_1 \cap L_2$. Thus, $ab \in P$. If $b \in P$, then $b^sN \subseteq L_1$ which is a contradiction. Hence $a \in P$ and so $a \in  \sqrt{(L_2 :_R N)}$.
Let $P, Q \in Coass_R(N)$. Then there exist completely irreducible submodules $L_1$ and $L_2$ of $M$ such that $P=(L_1:_RN)$ and $Q=(L_2:_RN)$. Thus, $P=\sqrt{(L_1:_RN)}$ and $Q=\sqrt{(L_2:_RN)}$. Hence, either $P \subseteq Q$  or
$Q \subseteq P$ and this completes the proof.

(b)
 The proof is similar to that of part (a).
\end{proof}

In \cite[2.10]{pb12}, it is shown that, if $R$ be a Noetherian ring, $M$ a finitely generated multiplication $R$-module, $N$ a proper submodule of $M$ such that $Ass_R(M/N)$ is a totally ordered set, and $(N :_R M)$ is a 2-absorbing ideal of $R$, then $N$ is a 2-absorbing submodule of $M$. In the following theorem we see that some of this conditions are redundant.
\begin{thm}\label{t11.11}
Let $N$ be a submodule of a multiplication $R$-module $M$ such that $(N :_R M)$ is a 2-absorbing ideal of $R$. Then $N$ is a 2-absorbing submodule of $M$.
\end{thm}
\begin{proof}
As $(N :_R M)\not = R$, $N \not = M$. Let $a, b \in R$, $m \in M$,  and $abm \in N$. Since $M$ is a multiplication $R$-module, there exists an ideal $I$ of $R$ such that $Rm=IM$. Thus $abIM \subseteq N$. Hence, $abI \subseteq (N:_RM)$. Now by assumption, $ab\in  (N:_RM)$ or $aI \subseteq (N:_RM)$ or $bI \subseteq (N:_RM)$. Therefore, $ab\in  (N:_RM)$ or $aIM \subseteq N$ or $bIM\subseteq N$. Thus  $ab\in  (N:_RM)$ or $am \in N$ or $bm \in N$.
\end{proof}

An $R$-module $M$ is said to be a \emph{comultiplication module} if for every submodule $N$ of $M$ there exists an ideal $I$ of $R$ such that $N=(0:_MI)$, equivalently, for each submodule $N$ of $M$, we have $N=(0:_MAnn_R(N))$ \cite{AF07}.
\begin{thm}\label{t1.11}
Let $N$ be a submodule of a comultiplication $R$-module $M$. Then we have the following.
\begin{itemize}
  \item [(a)] If $Ann_R(N)$ is a 2-absorbing ideal of $R$, then $N$ is a strongly 2-absorbing second submodule of $M$. In particular, $N$ is a 2-absorbing second submodule of $M$.
  \item [(b)] If $M$ is a cocyclic module and $N$ is a 2-absorbing second submodule of $M$, then $N$ is a strongly 2-absorbing second submodule of $M$.
\end{itemize}
\end{thm}
\begin{proof}
(a) Let $a, b \in R$, $K$ be a submodule of $M$, and $abN\subseteq K$. Then we have $Ann_R(K)abN=0$. So by assumption, $Ann_R(K)aN=0$ or $Ann_R(K)bN=0$ or $abN=0$. If $abN=0$, we are done. If $Ann_R(K)aN=0$ or $Ann_R(K)bN=0$, then $Ann_R(K) \subseteq Ann_R(aN)$ or $Ann_R(K) \subseteq Ann_R(bN)$. Hence, $aN \subseteq K$ or $bN \subseteq K$ since $M$ is a comultiplication $R$-module.

(b) By Proposition \ref{c11.6}, $Ann_R(N)$ is a 2-absorbing ideal of $R$. Thus the result follows from part (a).
\end{proof}

The following example shows that Theorem \ref{t1.11} (a) is not satisfied in general.

\begin{ex}\label{e1.11}
 By \cite[3.9]{AF07}, the $\Bbb Z$-module $\Bbb Z$ is not a comultiplication $\Bbb Z$-module. The submodule $N=p\Bbb Z$ of $\Bbb Z$, where $p$ is a prime number, is not strongly 2-absorbing second submodule. But $Ann_{\Bbb Z}(p\Bbb Z)=0$ is a 2-absorbing ideal of $R$.
\end{ex}

For a submodule $N$ of an $R$-module $M$ the the \emph{second radical} (or \emph{\emph{second socle}}) of $N$ is defined  as the sum of all second submodules of $M$ contained in $N$ and it is denoted by $sec(N)$ (or $soc(N)$). In case $N$ does not contain any second submodule, the second radical of $N$ is defined to be $(0)$ (see \cite{CAS13} and \cite{AF11}).

\begin{thm}\label{p111.11}
Let $M$ be a finitely generated comultiplication $R$-module. If $N$ is a strongly 2-absorbing second submodule of $M$, then $sec(N)$ s a strongly 2-absorbing second submodule of $M$.
\end{thm}
\begin{proof}
Let  $N$ be a strongly 2-absorbing second submodule of $M$. By Proposition \ref{p1.6} (a), $Ann_R(N)$ is a 2-absorbing ideal of $R$. Thus by \cite[2.1]{Ba07}, $\sqrt{Ann_R(N)}$ is a 2-absorbing ideal of $R$. By \cite[2.12]{AF25}, $Ann_R(sec(N))=\sqrt{Ann_R(N)}$.  Therefore, $Ann_R(sec(N))$ is a 2-absorbing ideal of $R$. Now the result follows from Theorem \ref{t1.11} (a).
\end{proof}

\begin{lem}\label{l2.14}
Let $f : M \rightarrow \acute{M}$ be a monomorphism of R-modules.  If $\acute{L}$ is a completely irreducible submodule of $f(M)$, then $f^{-1}(\acute{L})$ is a completely irreducible submodule of $M$.
\end{lem}
\begin{proof}
This is strighatforward.
\end{proof}

\begin{lem}\label{l2.15}
Let $f : M \rightarrow \acute{M}$ be a monomorphism of R-modules.
If $L$ is a completely irreducible submodule of $M$, then $f(L)$ is a completely irreducible submodule of $f(M)$.
\end{lem}
\begin{proof}
Let $\{\acute{N}_i\}_{i \in I}$ be a family of submodules of $f(M)$ such that $f(L)= \cap_{i \in I}\acute{N}_i$. Then $L=f^{-1}f(L)=f^{-1}(\cap_{i \in I}\acute{N}_i)=\cap_{i \in I}f^{-1}(\acute{N}_i)$. This implies that there exists $i \in I$ such that $L=f^{-1}(\acute{N}_i)$ since $L$ is a completely irreducible submodule of $M$. Therefore, $f(L) =ff^{-1}(\acute{N}_i)=f(M) \cap \acute{N_i}=\acute{N}_i$, as requested.
\end{proof}

\begin{thm}\label{t2.16}
Let $f : M \rightarrow \acute{M}$ be a monomorphism of R-modules. Then we have the following.
\begin{itemize}
  \item [(a)] If $N$ is a strongly 2-absorbing second submodule of $M$, then $f(N)$ is a 2-absorbing second submodule of $\acute{M}$.
  \item [(b)] If $N$ is a 2-absorbing second submodule of $M$,  then $f(N)$ is a 2-absorbing second submodule of $f(M)$.
  \item [(c)] If $\acute{N}$ is a strongly 2-absorbing second submodule of $\acute{M}$ and $\acute{N} \subseteq f(M)$, then $f^{-1}(\acute{N})$ is a 2-absorbing second submodule of $M$.
  \item [(d)] If $\acute{N}$ is a 2-absorbing second submodule of $f(M)$, then $f^{-1}(\acute{N})$ is a 2-absorbing second submodule of $M$.
\end{itemize}
\end{thm}
\begin{proof}
(a) Since $N \not =0$ and $f$ is a monomorphism, we have $f(N) \not =0$. Let $a, b \in R$, $\acute{L}$ be a completely irreducible submodule of $\acute{M}$, and $abf(N)\subseteq \acute{L}$. Then $abN \subseteq f^{-1}(\acute{L})$. As $N$ is strongly 2-absorbing second submodule, $aN \subseteq f^{-1}(\acute{L})$ or $bN \subseteq f^{-1}(\acute{L})$ or $abN=0$. Therefore,
 $$
 af(N) \subseteq f(f^{-1}(\acute{L}))=f(M) \cap \acute{L} \subseteq \acute{L}
 $$
 or
 $$
 bf(N) \subseteq f(f^{-1}(\acute{L}))=f(M) \cap \acute{L} \subseteq \acute{L}
 $$
 or $abf(N)=0$, as needed.

 (b) This is similar to the part (a).

 (c) If $f^{-1}(\acute{N})=0$, then $f(M) \cap \acute{N}=ff^{-1}(\acute{N})=f(0)=0$. Thus $\acute{N}=0$, a contradiction. Therefore, $f^{-1}(\acute{N})\not=0$. Now let $a, b \in R$, $L$ be a completely irreducible submodule of $M$, and $abf^{-1}(\acute{N})\subseteq L$. Then
  $$
  ab\acute{N}=ab(f(M) \cap \acute{N})=abff^{-1}(\acute{N})\subseteq f(L).
  $$
 As $\acute{N}$ is strongly 2-absorbing second submodule, $a\acute{N} \subseteq f(L)$ or $b\acute{N} \subseteq f(L)$ or $ab\acute{N}=0$. Hence $af^{-1}(\acute{N}) \subseteq f^{-1}f(L)=L$ or $bf^{-1}(\acute{N}) \subseteq f^{-1}f(L)=L$ or $abf^{-1}(\acute{N})=0$, as desired.

 (d) By using Lemma \ref{l2.15}, this is similar to the part (c).
\end{proof}

\begin{cor}\label{t2.17}
Let $M$ be an $R$-module and $N\subseteq K$ be two submodules of $M$. Then we have the following.
\begin{itemize}
  \item [(a)] If $N$ is a strongly 2-absorbing second submodule of $K$, then $N$ is a 2-absorbing second submodule of $M$.
  \item [(b)] If $N$ is a strongly 2-absorbing second submodule of $M$, then $N$ is a 2-absorbing second submodule of $K$.
\end{itemize}
\end{cor}
\begin{proof}
This follows from Theorem \ref{t2.16} by using the natural monomorphism $K\rightarrow M$.
\end{proof}

A non-zero submodule $N$ of an $R$-module $M$ is said to be a \emph{weakly second submodule} of $M$ if $rsN\subseteq K$, where $r,s \in R$ and $K$ is a submodule of
$M$, implies either $rN\subseteq K$ or $sN\subseteq K$ \cite{AF101}.

\begin{prop}\label{p2.18}
Let $N$ ba a non-zero submodule of an $R$-module $M$. Then $N$ is a weakly second submodule of $M$ if and only if $N$ is a strongly 2-absorbing second submodule of $M$ and $Ann_R(N)$ is a prime ideal of $R$.
\end{prop}
\begin{proof}
Clearly, if  $N$ is a weakly second submodule of $M$, then $N$ is a strongly 2-absorbing second submodule of $M$ and by \cite[3.3]{AF101}, $Ann_R(N)$ is a prime ideal of $R$. For the converse, let $abN \subseteq H$ for some $a, b \in R$ and submodule $K$ of $M$ such that neither $aN \subseteq H$ nor $bN \subseteq H$. Then $ab \in Ann_R(N)$ and so either $a \in Ann_R(N)$ or $b \in Ann_R(N)$. This contradiction shows that $N$ is weakly second.
\end{proof}

The following example shows that the two concepts of strongly 2-absorbing second submodule and weakly second submodule are different in general.
\begin{ex}\label{e2.18}
Let $p$, $q$ be two prime numbers, $N=<1/p+\Bbb Z>$, and $K=<1/q+ \Bbb Z>$. Then $N\oplus K$ is not a weakly second submodule of the $\Bbb Z$-module $\Bbb Z_{p^\infty} \oplus \Bbb Z_{q^\infty}$. But $N\oplus K$ is a strongly 2-absorbing second submodule of the $\Bbb Z$-module $\Bbb Z_{p^\infty} \oplus \Bbb Z_{q^\infty}$.
\end{ex}

\begin{cor}\label{c2.18}
Let $N$ be a submodule of an $R$-module $M$. Then

\begin{picture}(10,65.75)(25,70)
\put(66.25,115.25){\makebox(270,0)[cc]{$N\\\ is\\\ second\\\ \Rightarrow
N\\\ is\\\ weakly\\\ second\\\ \Rightarrow N\\\ is\\\ strongly \\\ 2-absorbing\\\ second$}}
\put(64.25,95.25){\makebox(220,0)[cc]{$\Rightarrow N\\\ is\\\ 2-absorbing\\\ second$.}}
\end{picture}
\\
In general, none of the above implications is reversible.
\end{cor}
\begin{proof}
The first assertion follows from \cite[3.2]{AF101}, Proposition \ref{p2.18}, and Example \ref{e2.2}. The second assertion follows from \cite[3.2]{AF101}, Example \ref{e2.18}, and Example \ref{e2.2}.
\end{proof}

\begin{prop}\label{p8.21}
Let $M$ be an $R$-module and $\{K_i\}_{i \in I}$ be a chain of strongly
2-absorbing second submodules of $M$. Then $\cup_{i \in I}K_i$ is a strongly 2-absorbing second submodule of $M$.
\end{prop}
\begin{proof}
Let $a, b \in R$, $H$ be a submodule of $M$, and $ab(\cup_{i \in I}K_i) \subseteq H$. Assume that $a(\cup_{i \in I}K_i )\not \subseteq H$ and $b(\cup_{i \in I}K_i) \not \subseteq H$. Then there are $m,n \in I$, where $aK_n \not \subseteq H$ and $bK_m \not \subseteq H$. Hence, for every $K_n \subseteq K_s$ and $K_m \subseteq K_d$, we have that $aK_s \not \subseteq H$ and $bK_d \not \subseteq H$. Therefore, for each submodule $K_h$ such that $K_n \subseteq K_h$ and $K_m \subseteq K_h$ we have $abK_h=0$. Hence $ab(\cup_{i \in I}K_i)=0$, as needed.
\end{proof}

\begin{defn}\label{d8.22}
We say that a 2-absorbing second submodule $N$ of an $R$-module $M$
is a \emph {maximal strongly 2-absorbing second submodule} of a submodule
$K$ of $M$, if $N \subseteq K$ and there does not exist a strongly 2-absorbing second submodule $H$ of $M$ such that $N \subset H \subset K$.
\end{defn}

\begin{lem}\label{l8.23}
 Let $M$ be an $R$-module. Then every strongly 2-absorbing second submodule of $M$ is contained in a maximal strongly 2-absorbing second submodule of $M$.
\end{lem}
\begin{proof}
This is proved easily by using Zorn's Lemma and Proposition \ref{p8.21}.
\end{proof}

\begin{defn}\label{d8.24}
Let $N$ be a submodule of an $R$-module $M$.
We define the \emph{strongly 2-absorbing second radical} of $N$ as the sum of all strongly 2-absorbing second submodules of $M$ contained in $N$ and we denote it by $s.2.sec(N)$. In case $N$ does not contain any strongly 2-absorbing second submodule, the strongly 2-absorbing second radical of $N$ is defined to be $(0)$. We say that $N \not =0$ is a \emph{strongly 2-absorbing second radical
submodule of $M$} if $s.2.sec(N)=N$.
\end{defn}

\begin{prop}\label{p8.27}
Let $N$ and $K$ be two submodules of an $R$-module $M$. Then we have the following.
\begin{itemize}
  \item [(a)] If $N\subseteq K$, then $s.2.sec(N)\subseteq s.2.sec(K)$.
  \item [(b)] $s.2.sec(N)\subseteq N$.
  \item [(c)] $s.2.sec(s.2.sec(N))=s.2.sec(N)$.
  \item [(d)] $s.2.sec(N)+s.2.sec(K)\subseteq s.2.sec(N+K)$.
  \item [(e)] $s.2.sec(N\cap K)=s.2.sec(s.2.sec(N) \cap s.2.sec(K))$.
  \item [(g)] If $N+K=s.2.sec(N)+s.2.sec(K)$, then $s.2.sec(N+K)=N+K$.
\end{itemize}
\end{prop}
\begin{proof}
These are straightforward.
\end{proof}

\begin{cor}\label{l8.25}
Let $N$ be a submodule of an $R$-module $M$. If $s.2.sec(N)\not =0$, then
$s.2.sec(N)$ is a strongly 2-absorbing second radical submodule of $M$.
\end{cor}
\begin{proof}
This follows from Proposition \ref{p8.27} (c).
\end{proof}

\begin{thm}\label{t8.26} Let $M$ be an $R$-module. If $M$ satisfies
the descending chain condition on strongly 2-absorbing second radical submodules, then
every non-zero submodule of $M$ has only a finite number of maximal strongly
2-absorbing second submodules.
\end{thm}
\begin{proof}
Suppose that there exists a non-zero submodule $N$ of $M$
such that it has an infinite number of maximal strongly 2-absorbing second
submodules. Then $s.2.sec(N)$ is a strongly
2-absorbing second radical submodule of $M$ and $s.2.sec(N)$ has an infinite
number of maximal strongly 2-absorbing second submodules. Let $S$ be a strongly 2-absorbing second radical submodule of $M$ chosen minimal such that $S$ has an infinite
number of maximal strongly 2-absorbing second submodules. Then $S$ is not strongly 2-absorbing second. Thus there exist $r, t \in R$ and a submodule $L$ of $M$
such that $rtS\subseteq L$ but $rS \not \subseteq L$, $tS \not
\subseteq L$, and $rtS\not=0$. Let $V$ be a maximal strongly 2-absorbing second submodule of $M$ contained in $S$. Then $V \subseteq (L:_Sr)$ or $V \subseteq
(L:_St)$ or $V \subseteq (0:_Srt)$ so that $V \subseteq s.2.sec((K:_Sr))$ or $V \subseteq
s.2.sec((K:_St))$ or $V \subseteq s.2.sec((0:_Srt))$. By the choice of $S$, the modules $s.2.sec((K:_Sr))$, $s.2.sec((K:_St))$, and $s.2.sec((0:_Srt))$ have only finitely many maximal strongly 2-absorbing second submodules. Therefore, there is only a finite number of possibilities for the module $S$, which is a contradiction.
\end{proof}

\begin{cor}\label{c8.10}
Every Artinian $R$-module has only a finite number of maximal strongly 2-absorbing second submodules.
\end{cor}

\begin{thm}\label{l8.11} Let $M$ be an $R$-module. If $E$ is an injective $R$-module and $N$ is a 2-absorbing submodule of $M$ such that $Hom_R(M/N,E) \not =0$, then $Hom_R(M/N,E)$ is a strongly 2-absorbing second $R$-module.
\end{thm}
\begin{proof}
Let $r, s \in R$. Since $N$ is a 2-absorbing submodule
of $M$, we can assume that $(N:_Mrs)=(N:_Mr)$  or $(N:_Mrs)=M$.
Since $E$ is an injective $R$-module, by replacing $M$ with $M/N$ in \cite[3.13 (a)]{AF101}, we have $Hom_R(M/(N:_Mr), E)=rHom_R(M/N,E)$. Therefore,
$$
rsHom_R(M/N, E)=Hom_R(M/(N:_Mrs), E)=
$$
$$
Hom_R(M/(N:_Mr), E)=rHom_R(M/N,E)
$$
or
$$
rsHom_R(M/N, E)=Hom_R(M/(N:_Mrs), E)=
$$
$$
Hom_R(M/M, E)=0,
$$
as needed
\end{proof}

\begin{thm}\label{t8.13}
Let $M$ be a strongly 2-absorbing second $R$-module and $F$ be a right exact linear covariant functor over the category of $R$-modules. Then $F(M)$ is a strongly 2-absorbing second $R$-module if $F(M) \not =0$.
\end{thm}
\begin{proof}
 This follows from \cite[3.14]{AF101} and Theorem \ref{t1.5} $(c) \Leftrightarrow (d)$.
\end{proof}

\bibliographystyle{amsplain}

\end{document}